\documentclass[reqno,12pt]{amsart}
\makeatletter
\@namedef{subjclassname@2020}{\textup{2020} Mathematics Subject Classification}
\makeatother
\title[Mild HJB equations for stable noise in infinite dimensions]{Mild solutions of HJB equations associated with cylindrical stable L\'evy noise in infinite dimensions}
\author[Bondi]{Alessandro Bondi}
\address[Alessandro Bondi]
{Department of AI, Data and Decision Sciences, Luiss University, Rome, Italy}
\email{abondi@luiss.it}
\author[Gozzi]{Fausto Gozzi}
\address[Fausto Gozzi]
{Department of AI, Data and Decision Sciences, Luiss University, Rome, Italy}
\email{fgozzi@luiss.it}
\author[Priola]{Enrico Priola}
\address[Enrico Priola]{Department of Mathematics, University of Pavia, Pavia, Italy}
\email{enrico.priola@unipv}
\thanks{The third author is a member of GNAMPA}
\author[Zabczyk]{Jerzy Zabczyk}
\address[Jerzy Zabczyk]{Institute of Mathematics, Polish Academy of Sciences, Warsaw, Poland}
\email{jerzy@zabczyk.com}

\keywords{Hamilton-Jacobi-Bellman equations, stochastic PDEs with jumps, stochastic optimal control, dynamic programming, cylindrical stable L\'evy processes}
\subjclass[2020]{93E20, 35R15, 60G52 (primary); 60H15 (secondary)}

\date{\today}
\usepackage{amssymb,mathtools}
\usepackage{amsmath,amsthm}
\usepackage{amsfonts}
\usepackage{graphicx}
  \usepackage{color}
  \usepackage{comment,enumerate}
\RequirePackage[colorlinks,citecolor=blue,urlcolor=blue]{hyperref}
\makeatletter
\@namedef{subjclassname@2020}{\textup{2020} Mathematics Subject Classification}
\makeatother

 \hoffset= - 0.5 cm \voffset=  0.1 cm \newcommand{\upperromannumeral}[1]{\MakeUppercase{\romannumeral #1}}
\newtheorem{theorem}{Theorem}[section]
\newtheorem{lemma}[theorem]{Lemma}

\newtheorem{hypothesis}[theorem]{Hypothesis}
\newtheorem{remark}[theorem]{Remark}
\def\H{\mathcal H}
\def\P{{\mathbb P}}
\def\R{{\mathbb R}}

\def\E{{\mathbb E }}

\definecolor{gre}{rgb}{0.03,0.50,0.03}

\definecolor{darkviolet}{rgb}{0.58, 0.0, 0.83}

\def\hh{{\vskip 2mm \noindent }}

\begin{document}

\maketitle

\begin{center}
\emph{Dedicated to the memory of Giuseppe Da Prato}  
\end{center}

\begin{abstract}
	We study the optimal control of an infinite-dimensional stochastic system governed by an SDE in a separable Hilbert space driven by cylindrical stable noise. We establish the existence and uniqueness of a mild solution to the associated HJB equation. This result forms the basis for the proof of the Verification Theorem, which is the subject of  ongoing	research and will provide a sufficient condition for optimality.
\end{abstract}

\bigskip

\tableofcontents

\section{Introduction}

Our  study is concerned with a  {stochastic} control system
\begin{equation}\label{intro1}
	dX_s  = (A X_s + F(X_s)) ds  + a_s ds
	+   dZ_s, \quad  s \ge t,  \; X_t  = x \in  H,
\end{equation}
on a Hilbert space $H$ where $A$ is a linear operator and  $F$ a Lipschitz  continuous and bounded transformation from $H$ into $H$. Moreover, $(a_s)_{s\ge 0}$ is a predictable control process with values in  {the closed ball $\{z\in H : |z|\le R\}$}. Perturbations are modeled by a stochastic process $Z$
of pure  jump L\'evy type. The ultimate goal is to find a control process which minimizes the cost functional
\begin{equation*}
	J(t,x,a) = \E \bigg [ \int_t^T \Big(g(X_s^{t,x,a}) + \frac{1}{2} |a_s|^2\Big)ds  \; + \; h(X_T^{t,x,a}) \bigg],
\end{equation*}
where $X^{t,x,a}_s,\,s\in[t,T],\,x\in H,$ is a solution of \eqref{intro1} corresponding to the control $(a_s)_{s\ge 0}.$

To solve the control problem we apply the dynamic programming approach, with the nonlocal parabolic Hamilton-Jacobi-Bellman equation
\begin{align}
	\begin{aligned}\label{intro2}
		&	\begin{aligned}
			\partial_t u(t,x) = &g(x) +
			\langle Ax+F(x), Du(t,x)\rangle
			\\&	+\int_{H}\{
			u(t,x+y)-u(t,x)-\langle
			Du(t,x), y
			\rangle
			\}\nu (dy)
			\\&+\inf_{|\lambda|\le R} \Big[ \langle \lambda, Du(t,x)\rangle +\frac{1}{2}|\lambda|^2\Big]
			,\quad  t \in ]0,T],
		\end{aligned}
		\\&
		u(0,x) = h(x),\quad  x \in H,
	\end{aligned}
\end{align}
for the value function
\[
V(t,x) = \inf_{a \in {\mathcal U}} J(t,x,a)
\]
playing a central role. In Equation \eqref{intro2}, $\nu$ is the so-called intensity measure of the process $Z$; $\mathcal{U}$ denotes the set of all control processes.

The specific results will be formulated under the additional assumption that  {$A$ is an unbounded, negative definite, self-adjoint operator on $H$ having inverse $A^{-1}$ which is compact. This allows to cover the case when $Z$ is a  cylindrical $\alpha-$stable process with $\alpha\in(1,2)$ formally given by
\[
	Z_t = \sum_{n \ge 1}  Z_t^n e_n,\quad  t \ge 0,
\]
where $(e_n)_n$ is an
orthonormal basis of eigenfunctions of $A$ and
$(Z^n_t)_n$ a sequence of independent one-dimensional $\alpha-$stable L\'evy
processes (see Section \ref{sec2} for more details)}. These processes and the corresponding semilinear SPDEs are introduced in \cite{PZ11} and further investigated in, for instance, \cite{PSXZ}.

Note that the study of a control problem like ours has been initiated in \cite{CDP}  (see also the later paper \cite{Gozzi95}) in the  well-known cylindrical Wiener case, i.e., when $(Z^n_t)_n$ are independent one-dimensional Wiener processes.

The solution  will be achieved in two steps. \\
The first one is a proof of the existence and uniqueness of the so-called mild version of Equation \eqref{intro2} (see Equation \eqref{intro4} below). This is done in the present paper. \\ {The second step will be the subject of a future research, which will focus on the  proof of the \emph{fundamental formula}}:
\begin{align}\label{fond_intro}
	\begin{aligned}
		u(T-t,x)=J(t,x,a)
		+\mathbb{E}\bigg[&
		\int_{t}^{T}\!\bigg(\!
		\inf_{|\lambda|\le R} \Big[ \langle \lambda, Du(T-s,X^{t,x,a}_s)\rangle +\frac{1}{2}|\lambda|^2\Big]
		\\
		&	-\frac{1}{2}|a_s|^2\!-\!
		\langle Du(T-{s},X^{t,x,a}_{s}),    a_{s}  \rangle\bigg) ds\bigg]
		,
	\end{aligned}
\end{align}
valid for the mild solution $u$ of \eqref{intro2} and an arbitrary control process $(a_s)_{s\ge0}.$
It leads directly to the solution of the control problem. \\
{It is worth noting that our approach relies on the smoothing effect and gradient estimates of the transition semigroup of the Ornstein-Uhlenbeck (OU, for short) process associated with the random perturbation
	$Z$. This allows us to avoid the theory of viscosity solutions, which is particularly delicate for infinite-dimensional pure jump Lévy processes. Such theory requires restrictive conditions on the drift $F$ (see \cite{SZ0, SZ1, SZ2}). Moreover, it does not  cover the cylindrical Lévy case we consider. We also mention that it is currently an open problem to establish the regularity of such viscosity solutions.

	For another example of  infinite-dimensional pure jump Lévy process where the strong Feller and regularizing properties of the corresponding OU  semigroup are known, we refer to \cite{BO_iso}; see also Remark \ref{rem_intro} for more details.}

To define the mild version of \eqref{intro2}, denote by $P_t,\,t\ge 0,$ the transition semigroup of the generalized OU process $Z_A$ (cf. \cite{PesZ}):
\begin{equation}\label{intro3}
	dZ_{A,t}=AZ_{A,t}dt + d Z_t,\quad Z_{A,0}=x.
\end{equation}
For $\phi$ in the space of bounded continuous functions $C_b(H)$:
\[
P_t\phi(x)=\mathbb{E}\left[\phi(Z^x_{A,t})\right],\quad t\ge0,\,x\in H,
\]
where $Z^x_{A}$ is the mild solution of \eqref{intro3}. The mild version of \eqref{intro2} is of the form
\begin{equation}\label{intro4}
	u(t,x) = P_t h (x) + \int_0^t P_{t-s} [\H (\cdot ,
	D u(s, \cdot ))]  (x)\, ds ,\quad t
	\in [0,T],\; x \in H,
\end{equation}
where, for arbitrary $y\in H$ and $p\in H$,
\[
\mathcal {H} (y,p) =  g(y)+ \langle F(y), p \rangle +  \inf_{|\lambda|\le R} \Big[\langle \lambda, p \rangle
+ \frac{1}{2} |\lambda|^2 \Big].
\]
{  The existence and uniqueness of a regular  solution $u$ to \eqref{intro4} in the space $C^1_\gamma(H)$  {(see \eqref{eq:defC1gamma} for its definition)} is proved in Theorem \ref{thm_wp}. In particular,
such solution $u$ is  Fréchet differentiable in the space variable $x$.
In Theorem \ref{rit}
 we show in addition that the Fréchet derivative $Du(t, \cdot)$  is  $\theta-$H\"older continuous from $H$ into $H$ for suitable  $\theta \in (0,1)$, $t \in ]0,T]$. These theorems are
 the main results of this   paper  and are
crucial for the proof of the fundamental formula \eqref{fond_intro}.}

 Even in the finite-dimensional setting, contrary to the Wiener case, which is extensively analyzed, for instance, in the monograph \cite{Kry}, the theory of optimal
stochastic control problems with random perturbations of L\'evy type is not very developed (especially for the case of multiplicative L\'evy noises). {We refer to  \cite{touzi, ISHI, praga} for Bellman's principles involving special dynamics and to the more recent \cite{BP} for a dynamic programming principle associated with
more
 general controlled jump diffusions. We also mention the book \cite{OKS}, where several examples of control problems for  jump diffusions motivated by applications can be found.}

Control systems in infinite dimensions with Wiener-type perturbations are of current interest and discussed in many publications, see, in particular, the comprehensive monograph \cite{FGS17}.
\\
Besides the already mentioned \cite{SZ1,SZ2}, however, we are unaware of works on stochastic infinite-dimensional control systems with pure L\'evy-type perturbations without Gaussian component. They require essential modifications of the classical techniques, although basic dynamic programming ideas are still applicable.

{\begin{remark}\label{rem_intro}
	The paper  \cite{BO_iso} studies  regularizing properties and establishes gradient estimates for the OU transition semigroup corresponding to subordinated cylindrical Wiener processes $W_S$ formally given by
	\[
		W_{S_t}=\sum_{n \ge 1}B^n_{S_t}e_n,\quad t\ge0.
	\]
	Here, $(B_t^n)_n$ is a sequence of one-dimensional independent Brownian motions and $S$ is an independent $\alpha-$stable subordinator, with $\alpha\in (\frac12,1)$. The perturbation $W_S$ is $2\alpha-$stable and, unlike $Z$, is isotropic, i.e., invariant by rotation. Employing the gradient estimates in \cite{BO_iso}, the machinery devised in this paper can be adapted to analyze a state equation like \eqref{intro1} driven by $W_S$ instead of $Z$.
\end{remark}}

\section{Preliminary material}\label{sec2}

{\color{black} In this section, we recall some properties  of cylindrical $\alpha-$stable L\'evy processes from \cite{PZ11}. Moreover, in Subsection \ref{sub_state} we introduce the state equation of the control problem investigated in the following sections.
\subsection{On stable L\'evy processes in Hilbert spaces}\label{sub_prel}

 Let $H$ be a   real separable Hilbert space.
Given a stochastic basis $(\Omega, {\mathcal F}, ({\mathcal F}_t), \P)$
 satisfying the usual assumptions,
we consider a cylindrical  $\alpha-$stable process $Z=(Z_t)_t$,
 $\alpha \in (1,2)$, formally given by
$$
 Z_t = \sum_{n \ge 1} \beta_n Z_t^n e_n,\quad  t \ge 0.
$$
 Here,  $(e_n)_n$ is a    fixed
 orthonormal basis in $H$,
  $(\beta_n)_n\subset (0,\infty)$ is a sequence of {positive} numbers and
$(Z^n_t)_n$ is  a sequence of independent one-dimensional $\alpha-$stable L\'evy
 processes   defined on the previous
stochastic basis such that, for any $n \in {\mathbb N}$ and $t \ge 0$,
\begin{equation}\label{esagera}
	\E [ e^{i Z^n_t h}] = e^{- t|h|^{\alpha}},\quad  h \in \R.
\end{equation}
Let $A\colon D(A) \subset H \to H$ be an operator that fulfills the next assumptions taken from \cite{PZ11}.
\begin{hypothesis} \label{ciao1}
The operator $A$ is self-adjoint. Moreover, the following holds.
\begin{enumerate}[(i)]
 \item The reference basis $(e_n)_n$ is a basis of eigenvectors of $A$. More specifically, $(e_n)_n\subset D(A)$ and there exists a sequence $(\gamma_n)_n\subset (0,\infty)$ of positive numbers such that  $\gamma_n \to
 \infty $ as $n\to\infty$ and $$Ae_n=-\gamma_ne_n,\quad n\in \mathbb{N}.$$

\item There exists  $\gamma \in [1/\alpha,1) $ and
  $\bar{C} >0$ such that
 \begin{equation} \label{pu1}
 \beta_n \ge \bar{C} \,
 \gamma_n ^{\frac{1}{\alpha} - \gamma},\quad  n \in\mathbb{N}.
\end{equation}
\item The series  $\sum_{n \ge 1} {\beta_n^{\alpha} }
 \gamma_n^{-1} $ converges.
 \end{enumerate}
\end{hypothesis}
{ Note that when  $\sum_{n \ge 1}
 \gamma_n^{-1} $ converges we can  cover in particular  the cylindrical case $\sum_{n \ge 1}  Z_t^n e_n$. This happens, for instance, when $A $  is the generator of the one-dimensional heat semigroup on a bounded interval  with Dirichlet boundary conditions.}

For every $x\in H$,  the OU process $Z_A^{x}=(Z_{A,t}^{x})_t$ associated with $Z$ is defined by
\[
	Z^x_{A,t}=e^{tA}x +\sum_{n=1}^{\infty} \beta_n\int_{0}^{t}e^{-\gamma_n(t-s)}dZ^n_s,\quad t\ge 0.
\]
Thanks to Hypothesis \ref{ciao1}, by \cite[Proposition 4.2]{PZ11}, $Z_A^x$ is an $H-$valued process. Moreover, by \cite[Theorem 4.4]{PZ11}, we can consider a version of $Z_A^x$ which   is  stochastically continuous, predictable and with $p-$locally integrable trajectories, for every $p\in [1,\alpha).$ We observe that for our arguments we do not  need the c\`adl\`ag regularity for the paths of the OU processes. \hh
We denote by $C_b(H)$ [resp., $C_b(H;H)$] the Banach space of bounded and continuous real-valued [resp., $H-$valued] functions defined in $H$, endowed with the usual norm $\|\cdot\|_0$.  Additionally, $C^1_b(H)$ is the Banach space of continuous,  bounded and Fr\'echet
differentiable functions from $H$ into $\R$ with continuous and bounded Fr\'echet derivative, endowed with the usual norm $\|\!\cdot\!\|_1$.
Let $P=(P_t)_{t\ge 0}$ be the OU transition semigroup associated with the processes $Z_A^x,\,x\in H,$ i.e.,
\begin{equation}\label{def_semi}
	P_t\phi(x)=\mathbb{E}\left[\phi(Z^x_{A,t})\right],\quad t\ge0,\,x\in H,\,\phi\in C_b(H).
\end{equation}
In the sequel, we denote by $C=C(\alpha, \gamma)>0$ a positive constant allowed to change from line to line.
According to \cite[Theorem 4.14]{PZ11},  for every $t>0$, $P_t$ maps Borel measurable and bounded
 functions $\phi$
 into $C^1_b(H)$, and the following gradient estimate holds:
 \begin{equation}\label{grad_est}
 \|DP_t \phi\|_0=	\sup_{x \in H}| DP_t \phi(x)|
 	 \le \frac{ C}{t^{\gamma}}  \| \phi\|_{0}.
 \end{equation}
If additionally $\phi\in C_b(H)$, then the  G\^ateaux derivative of $P_t\phi$ is given by
 \begin{equation*}
 	\langle DP_t \phi (x), h \rangle  =
 	\int_H \phi(e^{tA} x +  y ) \, \,  J_t(h,y)
 	\; \mu_t (dy),\quad h\in H,
 \end{equation*}
 where $\mu_t$ is the law of the stochastic convolution $Z_{A,t}^0$ and
 $J_t(h,\cdot ) \in L^2(H, \mu_t)$ such that
 \begin{equation*}
 	\int_H  | J_t(h,y)|^2
 	\; \mu_t (dy) \le   \frac{ C}{t^{2\gamma}} |h|^2.
 \end{equation*}
 We refer to \cite[Theorem 7 and Corollary 8]{BO_iso} for analogous results in the context of OU processes driven  by
 subordinated cylindrical Brownian noises.

\subsection{The State Equation}\label{sub_state}
We consider  a map
\begin{equation}\label{f3}
	F: H \to H \;\; \text{ Lipschitz continuous and bounded};
\end{equation}
we denote by $[F]_\text{Lip}$ the Lipschitz continuity constant of $F$. For a fixed $R>0$,  we  define the set of processes
\begin{equation}\label{admissible_controls}
\mathcal{U}=\mathcal{U}_R=\left\{
a\colon [0,\infty)\times \Omega \to B_R(H) \text{ s.t. $a$ is predictable}
\right\},
\end{equation}
where $B_R(H)$ is the closed ball centered at $0$ with radius $R$ in $H$.\hh
Inspired by the cylindrical Wiener case analyzed  in \cite{CDP, Gozzi95}, for every $a\in \mathcal{U}$, $t\ge0$ and $x\in H$
we  consider  the following controlled nonlinear stochastic differential equation:
 \begin{equation}\label{ab1}
dX_s  = (A X_s + F(X_s)) ds  + a_s ds
 +   dZ_s, \quad  s \ge t,  \; X_t  = x \in  H.
\end{equation}
Starting from Section \ref{sec_control} below, we investigate a control problem featuring  \eqref{ab1} as the \emph{state equation}.
Indeed, the mild formulation of \eqref{ab1} has a pathwise unique solution, as shown in the next lemma, which is proved similarly to \cite[Theorem 5.3]{PZ11}.

\begin{lemma}\label{well_SPDE}
	For every $a\in \mathcal{U}$, $t\ge0$ and $x\in H$, \eqref{ab1} admits a unique predictable mild solution with $p-$locally integrable paths for $p\in[1,\alpha)$, that is, there exists a unique predictable
	process $X=(X_s)_{s\ge t}$ with trajectories in $L^p_{\text{loc}}(s,\infty)$ for any $p\in[1,\alpha)$ such that, $\P-$a.s.,
	\begin{equation}\label{mild_sol}
		X_s=e^{(s-t)A}x+\int_{t}^{s}e^{(s-r)A}\left(F(X_r)+a_r\right)dr + Z^0_{A,s}-e^{(s-t)A}Z^{0}_{A,t} ,\quad s\ge t.
	\end{equation}
\end{lemma}
\begin{proof}
	Let $a\in \mathcal{U}$, $t\ge0$ and $x\in H$; notice that a process $X=(X_s)_{s\ge t}$ satisfies \eqref{mild_sol} if and only if the process $Y=(Y_s)_{s\ge t}$ defined by $Y_s=X_s-Z^0_{A,s}+e^{(s-t)A}Z^0_{A,t}$ fulfills
	\begin{equation}\label{Y}
				Y_s=e^{(s-t)A}x+\int_{t}^{s}e^{(s-r)A}\left(F(Y_r+Z^0_{A,r}-e^{(r-t)A}Z^{0}_{A,t})+a_r\right)dr ,\quad s\ge t.
	\end{equation}
	We then focus on this equation and demonstrate that it admits a unique predictable solution with continuous trajectories. This is sufficient to deduce the properties of $X$  in the statement of this lemma.
	
	Fix $T>0$ and denote by $C([t,t+T];H)$ the Banach space of $H-$valued continuous functions on $[t,t+T]$ endowed with the usual supremum norm $\|\cdot\|_{0}$. For $\P-\text{a.e. } \omega\in \Omega$, the functional ${\Gamma_{t^{T}_0,\omega}}$ given by
	\begin{align*}
&	\Gamma_{t^{T}_0,\omega} (f)(s)	=e^{(s-t)A}x\\&\hspace{2.5cm}
+\int_{t}^{s}e^{(s-r)A}\!\left(F(f(r)+Z^0_{A,r}(\omega)-e^{(r-t)A}Z^{0}_{A,t}(\omega))+a_r(\omega)\right)\!dr, \\&\qquad   s\in [t,t+T],\, f\in C([t,t+T];H),
	\end{align*}
	is well defined. Indeed, recalling that $a_r\in B_R(H)$ for every $r\in [t,t+T]$ and $Z^{0}_{A,\cdot}\in L^{p}(t,t+T)$ for any $p\in [1,\alpha)$, by \eqref{f3} we have
	\begin{align*}
		|	\Gamma_{t_0^T,\omega} (f)(s)|\!&\le |x|\!+\!\int_{t}^{s}\!\!\!\left(|F(0)|\!+\!|a_r(\omega)|\!+\![F]_{\text{Lip}}\left(\|f\|_{0}+|Z^{0}_{A,r}(\omega)|\!+\!|Z^{0}_{A,t}(\omega)|\right)\right) \!dr\\&
		<\infty,\quad s\in [t,t+T],
	\end{align*}
	where we also use the fact that $e^{\cdot A}=(e^{uA})_{u\ge 0}$ is a contraction semigroup on $H$.  Since $e^{\cdot A}$ is strongly continuous, by the dominated convergence theorem we infer that $\Gamma_{t_0^T,\omega}$ takes values in $C([t,t+T];H)$. 	 Moreover,
	for every $f_1,f_2\in C([t,t+T];H)$,
	\begin{equation}\label{rel_const}
	\|	\Gamma_{t_0^T,\omega}( f_1)-\Gamma_{t_0^T,\omega} (f_2)\|_{0}\le [F]_{\text{Lip}} |T|\|f_1-f_2\|_{0}.
	\end{equation}
	It then follows that, for $T$ sufficiently small, $\Gamma_{t_0^T,\omega}$ is a contraction in $C([t,t+T];H)$ which has a unique fixed point $\bar{f}_{t_0^T,\omega}$.\\
	Given that the relation among  constants in \eqref{rel_const} does not depend on the initial point $x$, a standard argument by steps based on the semigroup property of $e^{\cdot A}$ enable us to consider the entire  half-line $[t,\infty)$. More precisely, thanks to \eqref{rel_const},  for every $n\in\mathbb{N}$ we can define iteratively the contraction mappings $\Gamma_{t_{nT}^{(n+1)T},\omega}$ in $C([t+nT,t+(n+1)T];H)$ by
			\begin{align*}
			&	\Gamma_{t_{nT}^{(n+1)T},\omega} (f)(s)	=e^{(s-t-nT)A}\bar{f}_{t_{(n-1)T}^{nT},\omega}(t+nT)
			\\&\hspace{3.2cm}+\int_{t+nT}^{s}\!e^{(s-r)A}\!\left(F(f(r)+Z^0_{A,r}(\omega)-e^{(r-t)A}Z^{0}_{A,t}(\omega))+a_r(\omega)\right)\!dr, \\&\qquad   s\in [t+nT,t+(n+1)T],\, f\in C([t+nT,t+(n+1)T];H),
		\end{align*}
	where $\bar{f}_{t_{(n-1)T}^{nT},\omega}$ denotes the unique fixed point in $C([t+(n-1)T,t+nT];H)$ of $\Gamma_{t_{(n-1)T}^{nT},\omega}$.
	 Therefore, the function $f_\omega\in C([t,\infty);H)$ defined by
	 \[
	 	f_\omega(s)= \bar{f}_{t_{(n-1)T}^{nT},\omega}(s),\quad s\in[t+(n-1)T,t+nT],\,n\in \mathbb{N},
	 \]
	 is the unique continuous mapping in $[t,\infty)$ such that
	\begin{align*}
	&	f_{\omega}(s)	=e^{(s-t)A}x\!+\!\int_{t}^{s}\!e^{(s-r)A}\left(F(f_\omega(r)+Z^0_{A,r}(\omega)-e^{(r-t)A}Z^{0}_{A,t}(\omega))+a_r(\omega)\right)dr,\\&
	\qquad s\ge t.
	\end{align*}
Additionally, for every $n\in\mathbb{N}$, denoting by $\Gamma_{t_{(n-1)T}^{nT},\omega}^{(m)}$ the composition of $\Gamma_{t_{(n-1)T}^{nT},\omega}$ with itself $m-$times,   we have
	\[
		\lim_{m\to \infty}\Big\|\Gamma_{t_{(n-1)T}^{nT},\omega}^{(m)}(0)-f_\omega\big|_{[t+(n-1)T,t+nT]}\Big\|_{0}=0.
	\]
The process $Y=(Y_s)_{s\ge t}$ defined by $Y_s(\omega)=f_\omega(s)$ for $s\in[t,\infty)$ and
$\P-\text{a.s. } \omega \in \Omega$ is then the unique continuous solution of \eqref{Y}. 
Furthermore, arguing by induction, one can easily see that $$\big(\Gamma_{t_{(n-1)T}^{nT},\omega}^{(m)}(0)\big)_m\text{ is a sequence of
 predictable
processes, for every } n\in \mathbb{N}.$$
As a result,  $Y$ is predictable, as $Y_s(\omega)=\lim_{m\to\infty} \Gamma_{t_{(n-1)T}^{nT},\omega}^{(m)}(0)(s)$ for $\P-\text{a.s. } \omega \in \Omega$ and $s\in[t+{(n-1)T},t+nT
]$. The proof is now complete.
\end{proof}
Consider the unique solution $(X_s)_{s\ge t}$ of \eqref{mild_sol}.  Recalling  that $Z^0_A$ is predictable, if we define  $X_s=X_t$ for every $s\in[0,t)$,  then  the process $X=(X_s)_{s\ge 0}$ is predictable, as well.
In the sequel, to stress the dependence of $X$ on the starting point $x$, the initial time $t$ and the control $a$, we denote $X$ by $X^{t,x,a}=(X_s^{t,x,a})_{s\ge 0}$.
{\color{black}\begin{remark}
Lemma \ref{well_SPDE} still holds when $F$ is only Lipschitz continuous. Indeed, the boundedness of $F$ in \eqref{f3} is never used in its proof.
\end{remark} }

\section{The Control Problem and the associated HJB equation}\label{sec_control}

For a fixed $R>0$, we  consider a control problem where the set of admissible controls is $\mathcal{U}=\mathcal{U}_R$, see \eqref{admissible_controls}, and the state equation is  \eqref{ab1}. As discussed in Subsection \ref{sub_state}, for any $a\in\mathcal{U},
\,t\ge0$ and $x\in H$, \eqref{ab1} admits a unique predictable mild solution $X^{t,x,a}=(X_s^{t,x,a})_{s\ge 0}$ satisfying \eqref{mild_sol}.\hh
Given $h,\,g\in C_b(H)$ and a finite time-horizon $T>0$,  the  cost functional $J (t,x, a)$ that we investigate is
\begin{align}\label{def_J}
	\notag&J(t,x,a) = \E \bigg [ \int_t^T \Big(g(X_s^{t,x,a}) + \frac{1}{2} |a_s|^2\Big)ds  \; + \; h(X_T^{t,x,a}) \bigg],\\
	&\qquad t\in[0,T],\,x\in H,\,a \in \mathcal{U}.
\end{align}
The corresponding value function $V: [0,T] \times H \to \R$ is then defined by
\begin{align}\label{def_V}
 V(t,x) = \inf_{a \in {\mathcal U}} J(t,x,a).
\end{align}
We study this control problem  following  the Dynamic Programming Approach, focusing on the nonlocal
parabolic HJB equation (see, e.g.,  \cite{CDP}, \cite{FGS17} and \cite{Z99})
\begin{align} \label {hjb}
\begin{cases}
 \partial_t u(t,x) = g(x) + \inf_{\lambda \in B_R(H)} [ {\mathcal L}^{\lambda} u(t,x)],\quad  t \in ]0,T],
\\
 u(0,x) = h(x),\quad  x \in H.
\end{cases}
 \end{align}
Here, for a sufficiently regular cylindrical function $\phi$, $$
{\mathcal L}^{\lambda} \phi(x)  = L^{OU} \phi (x)+  \langle F(x), D \phi (x)\rangle +  \Big[\langle \lambda, D \phi(x) \rangle + \frac{1}{2}|\lambda|^2\Big],\quad x\in H,
$$ where, denoting by $\nu(d\xi)$ the L\'evy measure of the processes $(Z^n)_n$ (see also \cite[Theorem 31.5]{sato})
\begin{align}\label{LOU}\notag
 L^{OU} \phi (x) = &\langle Ax, D \phi (x)\rangle \\&+  \sum_{j=1}^{\infty}  \int_{\mathbb{R}}\left (\phi(x+ \beta_j\xi e_j)- \phi(x)-  \beta_j \xi \frac{\partial \phi}{\partial x_j}(x)\right) \nu(d\xi),\quad x\in H.
\end{align}
Notice that, by \eqref{esagera}, \cite[Theorems 14.3 (ii) and 14.15]{sato} and the L\'evy-Khintchine formula,
\begin{equation}\label{def_calpha}
\nu(d\xi)=\frac{c_\alpha}{|\xi|^{\alpha+1}}d\xi,\quad \text{where }\quad c_\alpha=\frac{1}{2}\left(-\Gamma(-\alpha)\cos\frac{\pi\alpha}{2}\right)^{-1},
\end{equation}
{\color{black} thus, with the change of variables $\xi_j'=\beta_j\xi$, we can rewrite \eqref{LOU} as follows:
\begin{align*}
	L^{OU} \phi (x) = &\langle Ax, D \phi (x)\rangle \\&+ c_\alpha \sum_{j=1}^{\infty}  \beta_j^\alpha \int_{\mathbb{R}}\left (\phi(x+ \xi e_j)- \phi(x)-  \xi \frac{\partial \phi}{\partial x_j}(x)\right)\frac{1}{|\xi|^{\alpha+1}} d\xi,\quad x\in H.
\end{align*}
}
   We now  introduce the   Hamiltonian function $\mathcal{H}$ defined by
\begin{align}\label{def_ham}
\notag\mathcal {H} (x,p) &=    \inf_{\lambda \in B_R(H)} \Big[\langle \lambda, p \rangle
 + \frac{1}{2} |\lambda|^2 \Big] + \langle F(x), p \rangle + g(x),
\\&
\eqqcolon H(p) +  \langle F(x), p \rangle + g(x),\quad x,\,p \in H.
\end{align}
By imposing first order conditions on the G\^ateaux differential of the convex map $\lambda\mapsto \langle \lambda, p \rangle
	+ \frac{1}{2} |\lambda|^2  $ and applying the Cauchy--Schwarz inequality, we  derive an explicit expression for $H$, namely
\begin{equation}\label{def_H}
	H(p)=\begin{cases}
		-\frac{1}{2}|p|^2,&|p|\le R,\\
		-R|p|+\frac12R^2,&|p|> R.
	\end{cases}
\end{equation}
Note also that $\mathcal{H}(\cdot,0)=g$.
Using $\mathcal{H}$, the HJB equation \eqref{hjb} can be rewritten as follows:
\begin{align}\label{HJB_H}
\begin{cases}
	\partial_t u(t,x) =  \H(x,  Du(t,x)) + L^{OU} u(t,x),\quad  t \in ]0,T],
	\\
	u(0,x) = h(x),\quad x \in H.
\end{cases} \end{align}In Section \ref{Sec_mild}, we study the well-posedness of \eqref{HJB_H} in a mild formulation.
\hh
We conclude this part with a lemma stating some properties of $\mathcal{H}$. Its proof, which relies on  the definition in \eqref{def_ham}, \eqref{def_H} and the fact that $F$ and $g$ are continuous and bounded, is straightforward and therefore  omitted.
\begin{lemma}\label{lip_H}
	The Hamiltonian $\mathcal{H}\colon H\times H\to \mathbb{R}$ is continuous in both variables. Furthermore, there exists a constant $L>0$ such that
	$$
	|\H(x,p) - \H(x,q)| \le L |p-q|,\quad p,\,q,\,x \in H.
	$$
	In particular,
	\begin{equation}\label{Ham_bound}
|\H(x,p)|  \le L |p| + \lVert g\rVert_0,\quad x,\,p\in H.
	\end{equation}
\end{lemma}




\section{Mild solutions of the HJB equation}\label{Sec_mild}
In this section, we study Equation \eqref{HJB_H} in mild form. Recall that a suitably regular map $u\colon [0,T]\times H\to \R$ is a \emph{mild solution} of \eqref{HJB_H} if $u$ fulfills the following convolution equation:
\begin{equation}
 \label{km}
 u(t,x) = P_t h (x) + \int_0^t P_{t-s} [\H (\cdot ,
 D u(s, \cdot ))]  (x)\, ds ,\quad t
\in [0,T],\; x \in H.
\end{equation}
To stress the dependence of $u$ on the given functions $g$ and $h$, in the sequel we can also write
$$
u(t,x) = u^{g,h}(t,x),\quad t \in [0,T],\, x \in H.
$$

We search for solutions of \eqref{km}
 in the functional space
 \begin{equation}
 C^{1}_{\gamma}(H)=\! \Bigg\{u\colon[0,T]\times H\to \mathbb{R} \text{ cont., bounded }\bigg|\,\, \begin{aligned}
 &u(t, \cdot) \in C_b^1(H),\,t\in ]0,T] \\&\!\!
  \sup_{t\in]0,T]} t^\gamma\|Du(t,\cdot)\|_0<\infty
 \end{aligned}
 \Bigg\},
 \label{eq:defC1gamma}
 \end{equation}
where $\gamma$ is given in Hypothesis \ref{ciao1}. As in \cite[Section 9.2]{Z99}  {(see also \cite{Bo_iter}, \cite[Section 9.5]{DZ} and \cite[Section 4.2.2]{FGS17}, where similar spaces are introduced)},
 we consider the norm
$$
\| u\|_{C^{1}_{\gamma}} = \sup_{t \in [0,T]} \| u(t,\cdot)\|_0 + \sup_{t \in ]0,T]} t^{\gamma} \|Du(t,\cdot)\|_0, \quad u\in C^1_\gamma(H);
$$
the couple $(C^1_\gamma(H), \|\cdot\|_{C^1_\gamma})$ constitutes a Banach space. \\
In Theorem \ref{thm_wp} below we demonstrate the well-posedness of \eqref{km} in $C^1_\gamma(H)$.  {We refer the reader to \cite[Theorem 2.1]{Bo_iter}, \cite[Theorem 9.38]{DZ}, \cite[Section 9.2]{Z99} and \cite[Section 4.4.1]{FGS17} for similar results in different settings.}


Before that, we present a preliminary lemma on the regularity of the map $(t,x)\mapsto P_t\phi (x)$ on $[0,\infty)\times H$, for a given $ \phi\in C_b(H)$.
\begin{lemma}\label{lemma_contpart}
	For every $\phi\in C_b(H)$, the function $(t,x)\mapsto P_t\phi(x)$ is continuous in $[0,\infty)\times H$. Furthermore, given a direction $p\in H$, the G\^ateaux derivative  $(t,x)\mapsto\langle DP_t\phi(x),p\rangle$ along $p$ is continuous in  $  (0,\infty)\times H$.
\end{lemma}
\begin{proof}
	Fix $\phi\in C_b(H)$. The joint continuity of the map $(t,x)\mapsto P_t\phi(x)$ in $[0,\infty)\times H$ is a consequence of  \cite[Lemma 2.1]{BRS} and the stochastic continuity of $Z_A^0$.
{\color{black}	\hh As regards the G\^ateaux derivative, notice that, for every $t>0$, by \eqref{def_semi},
	\[
		P_t\phi(x+y)=\E\left[\phi\left(e^{tA}x+e^{tA}y+Z^0_{A,t}\right)\right]=P_t(\phi(\cdot+e^{tA}y))(x),\quad x,y\in H.
	\]
	 Then, given $p\in H$, for every $\epsilon>0$, by the semigroup property of $P$ we  write
	\begin{align}\label{solo_dopo}
	\notag	\langle  D P_t\phi(x),p \rangle&=\lim_{h\to0}P_{t-\epsilon}\left(\frac{P_{\epsilon}(\phi(\cdot+he^{tA}p))-P_{\epsilon}\phi}{h}\right)(x)
	\\\notag
	&=
	\lim_{h\to0}P_{t-\epsilon}\left(\frac{P_{\epsilon}\phi(\cdot+he^{(t-\epsilon) A}p))-P_{\epsilon}\phi}{h}\right)(x)\\
	&=
P_{t-\epsilon}\langle DP_{\epsilon}\phi(\cdot), e^{(t-\epsilon) A}p\rangle (x),\quad t\ge \epsilon,\,x\in H.
	\end{align}
Here we use the dominated convergence theorem for the last equality, which can be applied because, by \eqref{grad_est}, the mean value theorem and the fact that $(e^{uA})_{u\ge0}$ is a contraction semigroup on $H$,
\[
	|P_{\epsilon}\phi(y+he^{(t-\epsilon) A}p))-P_{\epsilon}\phi(y)|\le C\frac{1}{\epsilon^\gamma} \|\phi\|_0|h||p|
,\quad y\in H.
\]
If we now consider $t\ge \epsilon,\,x\in H$ and two sequences $(t_n)_n\subset[\epsilon,\infty),\, (x_n)_n\subset H$ such that $t_n\to t$ and $x_n\to x$ as $n\to \infty$, then
\begin{align*}
		&|	\langle  D P_{t_n}\phi(x_n),p \rangle-\langle  D P_t\phi(x),p \rangle|\le
			|	P_{t_n-\epsilon}(\langle DP_{\epsilon}\phi(\cdot), e^{(t_n-\epsilon) A}p-e^{(t-\epsilon) A}p\rangle)(x_n)|
			\\
			&\qquad +\!
			|	P_{t_n-\epsilon}(\langle DP_{\epsilon}\phi(\cdot), e^{(t-\epsilon) A}p\rangle)(x_n)-	P_{t-\epsilon}(\langle DP_{\epsilon}\phi(\cdot), e^{(t-\epsilon) A}p\rangle)(x)|\!\eqqcolon\mathbf{\upperromannumeral{1}}_n+\mathbf{\upperromannumeral{2}}_n.
\end{align*}
Observe that, by \eqref{grad_est} and the strong continuity of the semigroup $(e^{uA})_{u\ge0}$ on $H$,
\[
	\mathbf{\upperromannumeral{1}}_n
	\le \frac{C}{\epsilon^\gamma}\|\phi\|_0\,\left|e^{(t_n-\epsilon) A}p-e^{(t-\epsilon) A}p\right|\underset{n\to\infty}{\longrightarrow}0,
\]
and, considering that $\langle DP_\epsilon\phi(\cdot),e^{(t-\epsilon)A}p\rangle \in C_b(H)$, by the first result of this lemma
$$\lim_{n \to \infty}\mathbf{\upperromannumeral{2}}_n=0.$$ Therefore we conclude that $(t,x)\mapsto \langle DP_t\phi(x),p\rangle$ is continuous in $[\epsilon,\infty)\times H$. Since $\epsilon>0$ is arbitrary, the proof is complete.
}
\end{proof}
\begin{theorem}\label{thm_wp} There exists a
 unique  solution $u = u^{g,h}$ of \eqref{km} in $C^{1}_{\gamma}(H)$.
\end{theorem}
{\color{black}\begin{proof}
 Define the mapping $S$ in $ C_\gamma^1(H)$ by
 \begin{align}\label{contraction_map}
 	&\notag
 	S (u) (t,x) = P_t h (x) +
 	\int_0^t P_{t-s} [\H (\cdot ,
 	D u(s, \cdot ))]  (x)\, ds,\\&\qquad t\in[0,T],\,x\in H,\,u\in C^1_\gamma(H).
 \end{align}
Observe that $S$ takes values in $C_\gamma^1(H)$. Indeed, since $h\in C_b(H)$, by \eqref{grad_est} and Lemma \ref{lemma_contpart},  the function $(t,x)\mapsto P_th(x)$ belongs to $C_\gamma^1(H)$. Additionally, for every
 $u\in C_\gamma^1(H)$, the map $\Gamma (u)(t,x)\coloneqq \int_0^t P_{t-s} [\H (\cdot ,
 D u(s, \cdot ))]  (x)\, ds$ is continuous on $[0,T]\times H$ by the dominated convergence theorem and Lemmas \ref{lip_H} -  \ref{lemma_contpart}. Furthermore, once  again, by the dominated converge theorem and \eqref{grad_est}, $\Gamma u(t,\cdot)\in C_b^1(H)$ for any $t\in]0,T]$, and, by \eqref{Ham_bound},
 \begin{align*}
 &	\|D\Gamma (u)(t, \cdot ) \|_0\le C\int_{0}^{t}\frac{1}{(t-s)^\gamma}\left(L\frac{\|u\|_{C_\gamma^1}}{s^\gamma}+\|g\|_0\right)ds
 	\\&\qquad \le
 	\frac{C}{1-\gamma}T^{1-\gamma}\left(4^\gamma L\|u\|_{C_\gamma^1}\frac{1}{t^\gamma}+\|g\|_0\right),\quad t\in]0,T],
 \end{align*}
where we use \cite[Equation (2.12)]{Bo_iter} for the last inequality.
 Thus,
$S :  C^{1}_{\gamma}(H)\to C^{1}_{\gamma}(H) $. Since, by estimates similar to those above and Lemma \ref{lip_H},
\begin{align*}
\|	S (u_1)-S(u_2)\|_{C_\gamma^1}\le \frac{T^{1-\gamma}}{1-\gamma}L(1 +4^\gamma C)\|u_1-u_2\|_{C_1^\gamma},\quad u_1,u_2\in C_\gamma^1(H),
\end{align*}
we deduce that $S$ is a contraction in $C_\gamma^1(H)$ for $T$ small enough. Consequently, \eqref{km}  admits a unique solution $u \in C^{1}_{\gamma}(H)$ for a sufficiently small $T$. \\
This conclusion continues to hold even in the case of a general $T$, which can be demonstrated  by a standard step method relying on the semigroup property of $P$.
The theorem  is now completely proved.
\end{proof} }
We conclude the paper with a  regularity result on the solution $u=u^{g,h}$ to \eqref{km} -- Theorem \ref{rit} --    that seems to be new  even in the limiting Brownian case $\alpha=2$.
We  focus on the H\"older continuity of $Du(t,\cdot),\,t\in ]0,T]$. For every $\theta\in (0,1]$, we denote by $C^{0,\theta}_b(H)$ [resp., $C^{0,\theta}_b(H;H)$] the space of $\R-$valued [resp., $H-$valued] bounded and $\theta-$H\"older continuous functions $l$ endowed with the usual norm
\begin{gather*}
	\| l\|_{C^{0,\theta}_b} \coloneqq \| l\|_0 +  [l]_{\theta},\quad \text{where }[l]_{\theta}=\sup_{x,y\in H, \,x\neq y}\frac{|l(x)-l(y)|}{|x-y|^\theta}.
\end{gather*}
In particular, $C^{0,1}_b(H)$ [resp., $C^{0,1}_b(H;H)$] is the space of $\R-$[resp., $H-$]valued Lipschitz continuous and bounded functions. As we have done for $F$ in \eqref{f3}, we  write  $[l]_{\text{Lip}}$ for $[l]_1$.
We also consider the space
\begin{align*}
&	C_b^{1,\theta}(H) \coloneqq \{ f \in C_b^1(H) \; \text{s.t. }  Df : H \to H \;\; \text{is $\theta-$H\"older continuous}\},\\&\qquad \text{endowed with the norm $\| f\|_{C_b^{1,\theta}} \coloneqq \| f\|_{1} + [Df]_{\theta}$.}
\end{align*}
{\color{black}
 \begin{theorem} \label{rit} For every $\theta\in (0,1)$ such that $\gamma+\theta\gamma<1$,  the
 unique  solution $u = u^{g,h}$ of \eqref{km} in $C^{1}_{\gamma}(H)$ satisfies
  \begin{equation}\label{theta}
 u(t, \cdot) \in C^{1,\theta}_b(H), \quad t \in ]0,T],
 \end{equation}
that is,
 the Fréchet derivative $Du(t, \cdot)$  is  $\theta-$H\"older continuous from $H$ into $H$. Furthermore, there exists a constant $\tilde{L}>0$ such that
 \begin{equation}\label{est_thetaconstant}
 [Du(t,\cdot)]_\theta\le \tilde{L}\frac{1}{t^{\gamma+\gamma\theta}}	,\quad t\in ]0,T].
 \end{equation}
\end{theorem}
\begin{proof} Fix $t\in]0,T]$ and consider $\theta\in(0,1)$ such that $\gamma+\theta\gamma<1$.
	 For every $\phi\in C_b(H)$ and $k,p\in H$, denoting by $D^2_{kp}P_t\phi$ the Gâteaux derivative of $\langle DP_t\phi(\cdot), p\rangle$ along the direction $k$, by the semigroup property of $P$ (see also \eqref{solo_dopo} in the proof of Lemma \ref{lemma_contpart}) we infer that
\begin{align}\label{second_der}
 D^2_{kp}P_t\phi(x)= \langle D^{} P_{t/2}( \langle D^{}_{} P_{t/2} \phi (\cdot),  e^{\frac{t}{2}A}p \rangle) (x),k\rangle,\quad x\in H.
\end{align}
Then, by \eqref{grad_est}, since the $(e^{uA})_{u\ge0}$ is a contraction semigroup on $H$,
\begin{align} \label{dee}
\|  D \langle DP_t \phi(\cdot),p\rangle \|_0 \le \frac{C}{t^{2\gamma}} \| \phi\|_0 |p|.
\end{align}
 By  \cite[Theorem 2.3.3]{DZsecond} (see also the monograph \cite{LU}), the following characterization for the interpolation space $(UC_b(H) , UC^1_b(H))_{\theta, \infty} $ holds:
\begin{equation*} \label{inter}
(UC_b(H) , UC^1_b(H))_{\theta, \infty} = C^{0,\theta}_b(H).
\end{equation*}
Here, $UC_b(H)$ is the Banach space of uniformly continuous bounded $\R-$valued maps, and $UC^1_b(H)$ is the space of functions in $UC_b(H)$ with uniformly continuous and bounded Fréchet derivative.
Since $\langle D P_t \phi(\cdot),p\rangle  \in UC_b^1 (H)$, \cite[Example 2.3.4]{DZsecond}, \eqref{grad_est} and \eqref{dee} ensure that, for a constant $C_1=C_1(\theta, \alpha,\gamma,T)>0$ allowed to change from line to line,
 \begin{align} \label{hol}\notag
 | \langle D P_t \phi (x) - DP_t \phi(y),p\rangle|&\le \|
 \langle DP_t\phi(\cdot),p\rangle
 \|_{C_b^{0,\theta}}
 |x-y|^{\theta}
 \\&\notag
 \le  {C_1}\frac{1}{t^{\gamma(1-\theta)}}\left(\frac1{t^{\gamma\theta}}
 +
 \frac{1}{t^{2\gamma\theta}}
 \right)|p|\| \phi\|_0|x-y|^{\theta}\\
 &\le {C_1}
 \frac{|p|}{t^{\gamma+\gamma\theta}}
 \| \phi\|_0|x-y|^{\theta}, \quad  x,y \in H.
\end{align}
 Differentiating  \eqref{km}, by \eqref{Ham_bound} and \eqref{hol} we compute, for every $x,y\in H$,
\begin{align}\label{final_theta}
	\notag& |\langle D u(t,x) - D u(t,y),p\rangle| \le   |\langle D P_t h (x) - D P_t h(y), p\rangle|\\
	&\notag
	\quad \quad + \int_0^t    |\langle D P_{t-s} [\H (\cdot ,
	D u(s, \cdot ))]  (x) - DP_{t-s} [\H (\cdot ,
	D u(s, \cdot ))]  (y), p \rangle|\, ds\\
	&\quad
	\le C_1\left(\int_0^t\left( L\frac{\|u\|_{C^1_\gamma}}{s^{\gamma}}+\|g\|_0\right) \frac{1}{(t-s)^{\gamma + \gamma\theta }} ds
+
	\frac{1}{t^{\gamma+\gamma\theta}}
	\| h\|_0\right)|p||x-y|^{\theta}.
\end{align}
Taking the supremum over $p\in H$ such that $|p|\le 1$, this estimate shows that $Du(t, \cdot)$ is $\theta-$H\"older continuous from $H$ into $H$  (i.e., \eqref{theta}). By \cite[Equation (2.12)]{Bo_iter}, \eqref{est_thetaconstant} readily follows from \eqref{final_theta}, as well. Given that $t\in]0,T]$ is arbitrary, the proof is complete.
\end{proof}

}

\vskip 4mm  \noindent {\bf Acknowledgments.}  The authors would like to thank Giuseppe Da Prato for his original work on control theory and SPDEs, which has greatly influenced them.

\end{document}